\documentclass[11pt]{amsart}

\usepackage[T1]{fontenc}
\usepackage[utf8]{inputenc}
\usepackage{lmodern}
\usepackage{amsmath,amssymb,amsthm,mathtools}
\usepackage{microtype}
\usepackage[colorlinks=true,linkcolor=blue,citecolor=blue,urlcolor=blue]{hyperref}
\usepackage[nameinlink,capitalise,noabbrev]{cleveref}

\usepackage{tikz}
\usepackage{tkz-tab}
\usepackage{tkz-fct} 

\allowdisplaybreaks
\numberwithin{equation}{section}

\newtheorem{theorem}{Theorem}[section]
\newtheorem{proposition}[theorem]{Proposition}
\newtheorem{lemma}[theorem]{Lemma}
\newtheorem{corollary}[theorem]{Corollary}
\theoremstyle{definition}
\newtheorem{definition}[theorem]{Definition}
\theoremstyle{remark}

\DeclareMathOperator{\dist}{dist}
\DeclareMathOperator{\sech}{sech}
\newcommand{\N}{\mathbb{N}}
\newcommand{\R}{\mathbb{R}}
\newcommand{\C}{\mathbb{C}}
\newcommand{\G}{\mathcal{G}}
\newcommand{\E}{\mathcal{E}}
\newcommand{\Ncal}{\mathcal{N}}
\newcommand{\dd}{\,\mathrm{d}}
\newcommand{\eps}{\varepsilon}
\newcommand{\norm}[1]{\lVert #1\rVert}
\newcommand{\abs}[1]{\lvert #1\rvert}

\title[Asymmetric minimal-mass blow-up]{Asymmetric minimal mass blow-up for the quintic NLS on a star graph}

\author{Mattia Buratti}
\address{Dipartimento di Matematica, Sapienza Universit\`a di Roma,
Piazzale Aldo Moro 5, I-00185 Roma, Italy}
\email{mattia.buratti@uniroma1.it}
\author{Fran\c{c}ois Genoud}
\address{\'Ecole Polytechnique F\'ed\'erale de Lausanne,
EPFL Station 4, 1015 Lausanne, Switzerland}
\email{francois.genoud@epfl.ch}
\author{Vladimir Georgiev}
\address{Dipartimento di Matematica, Universit\`a degli Studi di Pisa,
Largo Bruno Pontecorvo 5, 56127 Pisa, Italy}
\address{Faculty of Science and Engineering, Waseda University,
3-4-1 Okubo, Shinjuku-ku, Tokyo 169-8555, Japan}
\address{Institute of Mathematics and Informatics,
Bulgarian Academy of Sciences,
Acad. Georgi Bonchev Str., Block 8, 1113 Sofia, Bulgaria}
\email{vladimir.simeonov.gueorguiev@unipi.it}



\begin{document}

\begin{abstract}
This work is concerned with
the focusing quintic nonlinear Schr\"odinger equation on a
star graph with two edges, with a delta interaction of arbitrary strength at the
vertex. We construct a non-radial, minimal mass, finite time blow-up
solution, whose singularity forms at an arbitrarily prescribed interior point of one edge.
Our approach is based on a suitable cut-off of the standard pseudoconformal 
blow-up solution of the focusing quintic NLS on the line, inspired 
by Ogawa-Tsutsumi \cite{Ogawa90} and Burq-G\'erard-Tzvetkov \cite{burq2003singular}. 
A noteworthy feature of our 
construction is the persistence of the domain of the Hamiltonian under the flow,
obtained by the persistence of regularity method of Kato \cite{Kato1987}.
We also identify a two-parameter family of distinct blow-up profiles, which should play
a central role in the classification of non-radial minimal mass blow-up
solutions on the star graph.
Our construction is generic and can be extended to a star graph with 
an arbitrary number of edges. It can certainly also be adapted to other types of metric graphs.
\end{abstract}

\maketitle

\section{Introduction}

The focusing nonlinear Schr\"odinger equation (NLS) combines dispersive spreading
with nonlinear concentration, giving rise to both standing wave solutions and
finite time blow-up solutions. In this article, we continue the study and construction
of finite time blow-up solutions of the focusing quintic NLS on star graphs, 
initiated in \cite{genoud2024blowup}. NLS equations on metric graphs have
attracted considerable attention in recent years. To the best of our knowledge, however, 
blow-up solutions on graphs remain largely uncharted territory.
We refer to the introduction in 
\cite{genoud2024blowup} for an extensive discussion of the literature on blow-up
solutions of NLS equations and on other aspects of equations on metric graphs,
including standing wave solutions, variational structures, and stability.

Since a star graph is a one-dimensional structure, it is useful to recall
briefly some well-known results about the focusing quintic NLS on the line before
turning to the graph setting.
The focusing quintic NLS on the line,
\begin{equation}\label{eq:line-nls-intro}
 i\partial_t w+\partial_x^2w+\abs{w}^4w=0,
 \qquad w=w(t,x):\R\times\R \to \C,
\end{equation}
is $L^2$-critical, i.e.~the scaling symmetry of the equation,
\[
 w(t,x)\mapsto\lambda^{1/2}w(\lambda^2t,\lambda x),
 \qquad \lambda>0.
\]
preserves the $L^2$-norm. It has been known since 
Weinstein's seminal paper \cite{Weinstein1983} that this criticality gives rise to a 
mass threshold for the global existence of $H^1$ solutions. We define the mass
of $w$ as 
\[
M_\R(w(t))=\frac{1}{2}\|w(t,\cdot)\|_{L^2(\R)}^2,
\]
which is independent of $t$ if $w$ is a solution of \eqref{eq:line-nls-intro}.

It is clear by substitution that \eqref{eq:line-nls-intro} admits as a global solution 
the standing wave $w(t,x)=e^{i t}{Q}(x)$, where ${Q}$ is the unique 
positive even solution of the stationary equation
\begin{equation}\label{eq:ground-state-intro}
 {Q}''-{Q}+{Q}^5=0.
\end{equation}
This equation can be solved explicitly, yielding
\begin{equation} \label{eq:Q-explicit}
    {Q}(x)=3^{\frac14}\sech^{\frac12}(2x), \quad x\in\R.
\end{equation}

Every $H^1$ solution $w$ of \eqref{eq:line-nls-intro} whose initial datum 
$w_0=w(0,\cdot)$ satisfies
\[
M_\R(w_0)<M_\R({Q})
\]
is global in time. Weinstein proved this criterion for global existence
in \cite{Weinstein1983}, in fact for the $L^2$-critical focusing
NLS on $\R^N$ with any $N\geq1$. Furthermore, the criterion is sharp: there exists a 
minimal mass  blow-up solution, i.e.~a solution of~\eqref{eq:line-nls-intro} with mass 
$M_\R({Q})$ which blows up in finite time. It turns out this solution can be found 
explicitly as a pseudoconformal transform of the ground state standing wave $e^{i t}{Q}(x)$. 
Indeed, let
\begin{equation}\label{eq:PC}
S(t,x)=\frac{1}{\sqrt{|t|}} {Q}\Big(\frac{x}{|t|}\Big)e^{-i\frac{|x|^2}{4|t|}}e^{\frac{i}{|t|}}.
\end{equation}
Then $S$ is a solution of~\eqref{eq:line-nls-intro} defined for $t\in(-\infty,0)$ and we have
\begin{equation*}\label{eq:PC_speed}
\norm{S(t)}_{L^2(\R)}=\norm{Q}_{L^2(\R)},
\quad \norm{\partial_x S(t)}_{L^2(\R)}\mathop{\mbox{$\sim$}}\limits_{t \uparrow 0}\frac{C}{|t|}.
\end{equation*}
In particular, $S$ blows up at $t=0$ with the so-called \emph{pseudoconformal speed} 
$|t|^{-1}$. Furthermore, it was proved by Merle \cite{Merle1993} that, 
up to the symmetries of the equation, $S$ is the unique minimal mass blow-up solution 
of \eqref{eq:line-nls-intro}.

The above setting of the NLS on the line will henceforth be referred 
to as {\em Euclidean}.

Let us now describe the situation on the star graph, and present our main result.
As depicted in Figure~\ref{fig:star_graph}, a star graph $\G$ consists of a finite 
number $N\in\N$ of edges (half-lines) all connected at their common origin, which is the 
vertex of the graph. A function on the graph is denoted as $u=(u_1,\dots,u_N):\G\to\C$, 
where $u_j$ is defined on the $j$-th edge. Derivatives of each component can then be 
considered on the corresponding edge and the one-sided derivative at the vertex, 
$u_j'(0^+)$, is simply denoted by $u_j'(0)$. 

For $\alpha\in\R$,
we define the self-adjoint {\em Hamiltonian} 
$H_\alpha: D(H_\alpha)\subset L^2(\G) \to L^2(\G)$ by its action
\[
(H_\alpha u)_j:=-u_j''
\]
on each edge, subject to the continuity condition that all components have a common value
$u(0)$ at the vertex, together with the $\delta$ {\em interaction} condition
\begin{equation}\label{eq:delta_interaction-intro}
\sum_{j=1}^N u_j'(0) = \alpha u(0).
\end{equation}

The present work is devoted to the equation
\begin{equation}\label{eq:nls-intro}
 i\partial_tu-H_\alpha u+\abs{u}^4u=0, \qquad u=u(t,x):\R\times\G\to\C.
\end{equation}
The main conserved quantities associated with \eqref{eq:nls-intro} are the energy 
\begin{equation}
\label{eq:def_E}
E_\alpha(u):=\frac12 \norm{u_{ x}}_{L^2(\mathcal G)}^2+\frac\alpha 2| u(0)|^2
-\frac16\| u\|_{L^6(\mathcal G)}^6
\end{equation}
and the mass
\begin{equation}
\label{eq:def_M}
M(u):=\frac12\| u\|_{L^2(\mathcal G)}^2.
\end{equation}
The exact description of $\mathcal G$, the precise interpretation of~\eqref{eq:nls-intro},
and the definition of the usual function spaces on $\mathcal G$ will be given in 
Section~\ref{sec:setting}.

It has been shown in \cite{genoud2024blowup} that the Euclidean mass threshold 
$M_\R({Q})$ also determines the mass threshold for global existence of $H^1$
solutions of~\eqref{eq:nls-intro}. More precisely, if the initial datum $u_0=u(0,\cdot)$ verifies
\begin{equation}\label{eq:mass_threshold}
M(u_0)<\min\big\{\textstyle\frac N2,1\big\} M_\R({Q}),
\end{equation}
then the corresponding solution of~\eqref{eq:nls-intro} is global.

A function on $\G$ is called {\em radial} if all its components are equal, that is, 
if they all coincide with a single function $u:[0,\infty)\to\C$.
Let $Q_\G$ be the radial function on $\G$ consisting 
of a half-copy of ${Q}$ on each edge; that is, 
\[
(Q_\G)_j(x):= Q (x), \qquad x\geq0, \quad j=1,\dots,N.
\]
Then 
\begin{equation}\label{eq:min_mass}
M(Q_\G)=\textstyle\frac N2 M_\R({Q})
\end{equation}
and it was shown in \cite{genoud2024blowup} that an $H^1$ {\em radial solution} of 
\eqref{eq:nls-intro} is global if 
\begin{equation}\label{eq:mass_threshold_radial}
M(u_0)<M(Q_\G).
\end{equation}
Furthermore, in the attractive case $\alpha<0$, 
\cite{genoud2024blowup} constructs a radial finite time blow-up solution of \eqref{eq:nls-intro} 
with mass $M(Q_\G)$, showing that the mass threshold for global existence of $H^1$
radial solutions is sharp, as in the Euclidean setting.

For $N\geq 3$, the mass threshold \eqref{eq:mass_threshold} suggests the existence of
finite time blow-up solutions of \eqref{eq:nls-intro} with a mass as small as $M_\R({Q})$, which
in this case is strictly smaller than the radial threshold \eqref{eq:min_mass}. Thus,
such solutions would not be radial. Our main result, Theorem~\ref{thm:main}, 
shows that, indeed, there exists an asymmetric finite time blow-up solution with mass 
$M(u_0)=M_\R({Q})$. In view of criterion \eqref{eq:mass_threshold} for global existence,
this solution has minimal mass among blow-up solutions of \eqref{eq:nls-intro}.
Stripped of its technicalities, our theorem can be briefly formulated as follows.
\begin{theorem}\label[theorem]{thm:main-intro}
There exists a time $T>0$ and a non-radial solution 
\[
 u\in C([0,T);D(H_\alpha)) \cap C^1([0,T);L^2(\G))
\]
of \eqref{eq:nls-intro}, with mass $M(u)=M_\R({Q})$, 
which blows up at time $T$ with pseudoconformal speed:
\[
\|\partial_x u(t)\|_{L^2(\G)}\sim \displaystyle \frac{C}{T-t} \qquad(t\uparrow T).
\]
Furthermore, $u(t,x)$ concentrates all of its mass at a point $a>0$ located on edge number 
$1$, where it blows up with the pseudoconformal profile \eqref{eq:PC}. 
\end{theorem}

A common feature of \eqref{eq:nls-intro} and the Euclidean equation \eqref{eq:line-nls-intro} is
the one-dimensional nature of the problem, and the related $L^2$-criticality. But \eqref{eq:nls-intro}
differs substantially from the Euclidean case, in that several symmetries are broken. Namely,
translations, Galilean transformations, and pseudoconformal transformations 
do not preserve a fixed vertex condition,
while critical scaling changes the strength of the delta coupling. Thus, the explicit Euclidean 
pseudoconformal blow-up solution \eqref{eq:PC} cannot be transferred directly to the star graph.
Nevertheless, the main idea is that, as $t\uparrow T$, a copy of $S(t,x)$ centered at a point $a$ 
on one edge should asymptotically not be affected by the vertex condition, 
thanks to its exponential decay away from $a$.
We will thus carry out a perturbative construction using the pseudoconformal profile $S(t,x)$.
This line of attack is inspired by \cite{Ogawa90} and \cite{burq2003singular}, 
where critical NLS equations were considered on the line with periodic boundary 
conditions and in bounded domains of $\R^2$, respectively.

For simplicity, we prove our main result for a graph with two edges. The construction
is general and extends to an arbitrary finite number of edges with only minor modifications. 
Due to its perturbative nature, we believe
it is quite robust and could also be used for other types of graphs. 

We conclude this introduction with a short outline of our proof. 
Following \cite{Ogawa90} and \cite{burq2003singular}, we center a rescaled version of the Euclidean
pseudoconformal profile \eqref{eq:PC} at an interior point \(a>0\) of the first edge 
and multiply it by a cut-off function which vanishes near the vertex. This allows us to construct
an exact solution of \eqref{eq:nls-intro} backward in time from $t=T$ by applying the contraction
mapping principle to a final-time Duhamel map. This produces a solution behaving at leading order
on the first edge like the Euclidean pseudoconformal blow-up profile centered at point $a$. 
In the limit $t\uparrow T$, the convergence to the Euclidean profile, together with the conservation 
of the mass, ensures that our solution has mass $M_\R({Q})$. 

Since the vertex condition \eqref{eq:delta_interaction-intro} is not preserved by nonlinear
functions of $u$, 
it does not appear possible to perform the fixed point argument directly in 
$D(H_\alpha)$, which is a subspace of $H^2(\G)$. Instead, we apply the contraction mapping
principle in the form domain of $H_\alpha$, namely, $H^1(\G)$ with the continuity condition at the vertex, 
endowed with a suitable norm. The persistence
of $D(H_\alpha)$ under the flow of our solution is proved afterwards, using a functional-analytic approach
going back to Kato \cite{Kato1987}. 

The paper is organized as follows. In \cref{sec:setting}, we define
the graph and the associated functional setting, and state
our main result in full details. \Cref{sec:preliminaries} collects useful results about the
Hamiltonian $H_\alpha$ and the Euclidean pseudoconformal profile; it also introduces the cut-off profile 
and its error term.
\cref{sec:construction} gives the fixed point argument and properties of the solution. 
The persistence of $D(H_\alpha)$ under the flow is proved in \cref{sec:regularity}.
Finally, \cref{sec:terminal-family} identifies a two-parameter family of distinct blow-up profiles,
which provides a first step toward the classification of non-radial minimal mass blow-up solutions 
on the star graph.

\bigskip


\section{Setting and main result}\label{sec:setting}

A \emph{metric star graph} with $N$ edges consists of
a vertex $0$ to which are attacheded $N$ edges $e_1,\dots,e_N$ of infinite length. 
We identify each edge $e_j$ with the half-line $[0,\infty)$, 
the left endpoint $0$ corresponding to the vertex. 
A schematic representation of a star graph is given in Figure~\ref{fig:star_graph}. 

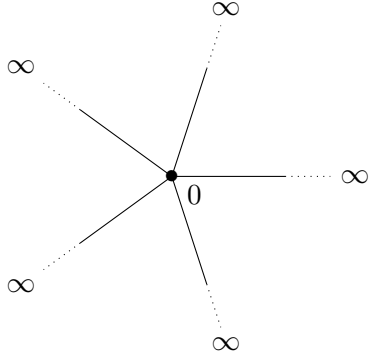
\begin{figure}[htbp!]
  \centering
\begin{tikzpicture}[scale=0.030]
  \node at (0,0) {$\bullet$};
  \node at (10,-8) {$0$};
    \foreach \i in {1,2,3,4,5}{
  \draw[rotate=72*\i] (0,0) -- (50,0) node[anchor={180+72*\i}] {};
    \draw[rotate=72*\i,dotted] (50,0) -- (70,0) node[anchor={180+72*\i}] {$\infty$};
  }
  \end{tikzpicture}    
  \caption{A star graph with $5$ edges.}
  \label{fig:star_graph}
\end{figure}

For simplicity, we carry out the construction of a blow-up solutions for a star graph with two edges, 
denoted by $\G$. Our proof extends to an arbitrary finite number of edges with only minor
modifications.

A function on $\G$ is written \(u=(u_1,u_2)\), where $u_1$ and $u_2$ are
functions defined on $[0,\infty)$, with values in $\C$.
For \(1\le p<\infty\) and $m\in\N$, the standard Lebesgue and Sobolev spaces on $\G$ are 
defined by
\[
 \norm{u}_{L^p(\G)}^p
 :=\sum_{j=1}^2\norm{u_j}_{L^p(\R_+)}^p,
 \qquad
 H^m(\G):=H^m(\R_+)\oplus H^m(\R_+).
\]
Derivatives at the vertex are the outward derivatives $u_1'(0^+)$ 
and $u_2'(0^+)$ and will simply be denoted by $u_1'(0)$ and $u_2'(0)$.

We will work in the energy space
\begin{equation}\label{eq:energy-space}
 \E(\G):=\{u\in H^1(\G):u_1(0)=u_2(0)\}.
\end{equation}
For \(\alpha\in\R\), we define an operator 
$H_\alpha: D(H_\alpha)\subset L^2(\G) \to L^2(\G)$ by
\begin{align}
 (H_\alpha u)_j&=-u_j'',\label{eq:H-action}\\
 D(H_\alpha)&=\left\{u\in H^2(\G):
 u_1(0)=u_2(0),\
 u_1'(0)+u_2'(0)=\alpha u_1(0)\right\}.
 \label{eq:H-domain}
\end{align}
The operator $H_\alpha$ is self-adjoint and the corresponding closed 
quadratic form is given by
\begin{equation}\label{eq:quadratic-form}
 q_\alpha(u)=\sum_{j=1}^2\int_0^\infty\abs{u_j'(x)}^2\dd x
 +\alpha\abs{u(0)}^2,
 \qquad D(q_\alpha)=\E(\G).
\end{equation}
The metric graph $\G$, endowed with the Hamiltonian $H_\alpha$,
is an example of a {\em quantum graph} (see e.g.~\cite{BerkKuch2013,Pava2019}).

We consider
\begin{equation}\label{eq:nls}
 i\partial_tu-H_\alpha u+F(u)=0,
 \qquad F(z):=\abs{z}^4z,
\end{equation}
where \(F\) acts componentwise. Two important conserved quantities associated 
with \eqref{eq:nls} are the mass and energy, respectively defined by
\begin{align}
 M(u)&:=\frac12\norm{u}_{L^2(\G)}^2,\label{eq:mass}\\
 E_\alpha(u)&:=\frac12q_\alpha(u)-\frac16\norm{u}_{L^6(\G)}^6.
 \label{eq:energy}
\end{align}

\begin{definition}\label[definition]{def:solutions}
Let \(I\subset\R\) be an interval.  A {\em mild solution} of \eqref{eq:nls} is a function
\(u\in C(I;\E(\G))\) such that, for \(t,t_0\in I\),
\begin{equation}\label{eq:duhamel-solution}
 u(t)=S_\alpha(t-t_0)u(t_0)
 +i\int_{t_0}^t S_\alpha(t-s)F(u(s))\dd s,
 \qquad S_\alpha(t):=e^{-itH_\alpha}.
\end{equation}
It is a {\em strong solution} if
\[
 u\in C(I;D(H_\alpha))\cap C^1(I;L^2(\G)),
\]
where \(D(H_\alpha)\) carries the graph norm
\(\norm f_{D(H_\alpha)}:=\norm f_2+\norm{H_\alpha f}_2\).
A solution {\em blows up} at \(T<\infty\) if its \(H^1(\G)\)-norm is
unbounded as \(t\uparrow T\).  Its {\em blow-up set} consists of the points in $\G$ 
for which no neighborhood has a uniform local \(H^1\)-bound.
\end{definition}

Recall that ${Q}$ is the unique positive even solution of \eqref{eq:ground-state-intro},
and $M_\R({Q})$ its mass. Our main result is the following theorem.

\begin{theorem}
\label[theorem]{thm:main}
Fix \(\alpha\in\R\) and \(a>0\) on the first edge.  There exist
\(\kappa>0\), \(\lambda_0\ge1\) and \(C,\delta>0\) such that, for every
\(\lambda\ge\lambda_0\), \eqref{eq:nls} has a strong solution
\[
 u_\lambda\in C([0,T_\lambda);D(H_\alpha))
 \cap C^1([0,T_\lambda);L^2(\G)),
 \qquad T_\lambda:=\frac{\kappa}{\lambda^2}.
\]
Let \(\mu(t):=\lambda(T_\lambda-t)\).  There is a real-valued
\(\chi\in C_c^\infty(0,\infty)\), which is equal to one in a neighborhood of \(a\), 
such that
\begin{equation}\label{eq:main-decomposition}
 u_\lambda=R_\lambda+v_\lambda,\qquad
 R_\lambda(t)=\bigl(\chi\widetilde R_\lambda(t),0\bigr),
\end{equation}
where
\begin{equation}\label{eq:main-profile}
 \widetilde R_\lambda(t,x)
 =\mu(t)^{-1/2}Q\!\left(\frac{x-a}{\mu(t)}\right)
 \exp\!\left(\frac{i}{\lambda\mu(t)}
 -\frac{i\lambda(x-a)^2}{4\mu(t)}\right)
\end{equation}
and $v_\lambda$ satisfies
\begin{equation}\label{eq:main-correction}
 \norm{v_\lambda(t)}_{L^2(\G)}
 \le \frac{C}{\lambda}e^{-\delta/(2\mu(t))},\qquad
 \norm{v_\lambda(t)}_{H^1(\G)}
 \le \frac{C}{\lambda}e^{-\delta/(3\mu(t))}.
\end{equation}
Moreover, $u_\lambda$ satisfies the following properties:
\begin{enumerate}
\item[(i)] \(M(u_\lambda(t))=M_\R(Q)\) for \(0\le t<T_\lambda\).
\item[(ii)] For every \(0<\eta<a/2\), with \(I_\eta=(a-\eta,a+\eta)\),
\begin{equation}\label{eq:local-gradient-asymptotic}
 \mu(t)\norm{\partial_xu_{\lambda,1}(t)}_{L^2(I_\eta)}
 \longrightarrow\norm{Q'}_{L^2(\R)}
 \qquad(t\uparrow T_\lambda).
\end{equation}
\item[(iii)] If \(K\Subset[0,\infty)\setminus\{a\}\) is a compact interval, then
\begin{equation}\label{eq:active-away}
 \norm{u_{\lambda,1}(t)}_{H^1(K)}\longrightarrow0
 \qquad(t\uparrow T_\lambda).
\end{equation}
\item[(iv)] On the second edge,
\begin{equation}\label{eq:inactive-edge}
 \norm{u_{\lambda,2}(t)}_{H^1(\R_+)}\longrightarrow0
 \qquad(t\uparrow T_\lambda).
\end{equation}
\item[(v)] For every \(\eps>0\),
\begin{equation}\label{eq:mass-concentration}
 \int_{\{x\ge0:\,\abs{x-a}\le\eps\}}
 \abs{u_{\lambda,1}(t,x)}^2\dd x
 \longrightarrow\norm{Q}_{L^2(\R)}^2
 \qquad(t\uparrow T_\lambda).
\end{equation}
\end{enumerate}
\end{theorem}

\section{Preliminaries}\label{sec:preliminaries}

In this section, we present useful properties of the Hamiltonian $H_\alpha$
(see \cite{Pava2019} for a systematic treatment),
some elements of well-posedness theory for \eqref{eq:nls} (see \cite{genoud2024blowup}
for more details), and properties of the Euclidean pseudoconformal
profile. We then construct the approximate profile and state its most important properties,
as well as those of the associated error term.

\begin{proposition}[Linear and nonlinear framework]\label[proposition]{prop:framework}
Let \(\alpha_-:=\max\{-\alpha,0\}\) and
\[
 c_\alpha:=1+\frac{\alpha_-^2}{4}.
\]

Since $\G$ has two edges,
\[
 \inf\sigma(H_\alpha)=-\frac{\alpha_-^2}{4},
 \qquad H_\alpha+c_\alpha I\ge I.
\]
In particular, $H_\alpha+c_\alpha I$ is strictly positive. Furthermore,
\(D((H_\alpha+c_\alpha)^{1/2})=\E(\G)\) and
\begin{equation}\label{eq:norm-equivalence}
 \norm{f}_{E_{c_\alpha}}
 :=\norm{(H_\alpha+c_\alpha)^{1/2}f}_{L^2(\G)}
 \simeq\norm{f}_{H^1(\G)}.
\end{equation}
The group \(S_\alpha(t)\) is unitary on \(L^2(\G)\), preserves
\(\E(\G)\) and satisfies
\begin{equation}\label{eq:group-H1}
 \norm{S_\alpha(t)f}_{H^1(\G)}
 \le C_\alpha\norm{f}_{H^1(\G)},
 \qquad t\in\R,\quad f\in\E(\G).
\end{equation}
For data in \(\E(\G)\), \eqref{eq:nls} is locally well-posed in the
mild sense and conserves \eqref{eq:mass}--\eqref{eq:energy}, with the usual
blow-up alternative. Furthermore, if the initial datum belongs to $D(H_\alpha)$,
the corresponding energy solution is strong throughout its
maximal interval of existence.
\end{proposition}

We shall use
\begin{align}
 \norm{f}_{L^\infty(\G)}
 &\le C\norm{f}_{H^1(\G)},\label{eq:sobolev-graph}\\
 \norm{f}_{L^q(\G)}
 &\le C_q\norm{f'}_2^{1/2-1/q}\norm{f}_2^{1/2+1/q},
 \qquad 2\le q<\infty.\notag
\end{align}
For \(F(z)=\abs{z}^4z\), regarded as a real differentiable map on
\(\C\simeq\R^2\),
\begin{align}
 \abs{F(z)-F(w)}
 &\le C(\abs{z}^4+\abs{w}^4)\abs{z-w},\label{eq:F-Lipschitz}\\
 \norm{DF(z)}&\le C\abs{z}^4,\qquad
 \norm{DF(z)-DF(w)}
 \le C(\abs{z}^3+\abs{w}^3)\abs{z-w}.
 \label{eq:DF-Lipschitz}
\end{align}

\begin{lemma}[Euclidean pseudoconformal profile]\label[lemma]{lem:euclidean-profile}
For \(a,T\in\R\), \(\ell>0\) and \(p,\vartheta\in\R\), define
\begin{equation}\label{eq:euclidean-family}
\begin{aligned}
 S_{a,T,\ell,p,\vartheta}(t,x)
 :={}&\frac{e^{i\vartheta}}{(\ell(T-t))^{1/2}}
 \exp\!\left(\frac{i}{\ell^2(T-t)}
 -\frac{i(x-a)^2}{4(T-t)}\right)\\
 &\times Q\!\left(\frac{x-a}{\ell(T-t)}-p\right),
 \qquad t<T.
\end{aligned}
\end{equation}
Then \(S_{a,T,\ell,p,\vartheta}\) solves \eqref{eq:line-nls-intro}, has
mass \(M_{\R}(Q)\) and concentrates at
\begin{equation}\label{eq:physical-center}
 x_c(t)=a+\ell(T-t)p.
\end{equation}
Moreover, \(Q\) and all its derivatives decay exponentially.
\end{lemma}

\begin{proof}
The properties of \(Q\) follow from \eqref{eq:Q-explicit}; uniqueness of the
positive ground state is classical \cite{Kwong1989}.  Put
\(\tau=T-t\), \(\mu=\ell\tau\) and \(z=(x-a)/\mu-p\).  Direct
differentiation of \eqref{eq:euclidean-family} gives
\[
 i\partial_tS+\partial_x^2S+\abs{S}^4S
 =\mu^{-1/2}e^{i\vartheta+i/(\ell^2\tau)-i(x-a)^2/(4\tau)}
 \frac{Q''-Q+Q^5}{\ell^2\tau^2}=0.
\]
The change of variables \(z=(x-a)/\mu-p\) gives
\(\frac12\norm{S}_{L^2(\R)}^2=M_{\R}(Q)\).
The maximum of \(\abs S\) occurs at \(z=0\), proving
\eqref{eq:physical-center}.
\end{proof}

Now that we have introduced the preliminary tools and described the pseudoconformal family, we proceed with the construction of the approximate profile on the graph using a smooth cut-off function.

Fix \(a>0\).  Choose a real-valued \(\chi\in C_c^\infty(0,\infty)\),
\(0\le\chi\le1\), such that
\begin{equation}\label{eq:cutoff}
 \chi=0\quad\text{on }(0,a/4),\qquad
 \chi=1\quad\text{on }(a/2,3a/2).
\end{equation}
Fix \(\kappa>0\), let \(\lambda\ge1\) and set
\begin{equation}\label{eq:scales}
 T=T_\lambda:=\frac{\kappa}{\lambda^2},\qquad
 \mu(t):=\lambda(T-t),\qquad 0\le t<T.
\end{equation}
The Euclidean profile is
\begin{equation}\label{eq:profile}
 \widetilde R_\lambda(t,x)
 :=\mu(t)^{-1/2}Q\!\left(\frac{x-a}{\mu(t)}\right)e^{i\theta_\lambda(t,x)},
 \qquad
 \theta_\lambda(t,x):=\frac{1}{\lambda\mu(t)}
 -\frac{\lambda(x-a)^2}{4\mu(t)}.
\end{equation}
Since $\widetilde R_\lambda=S_{a,T,\lambda,0,0}$, by \cref{lem:euclidean-profile}, it is an exact free quintic NLS solution.
Define
\begin{equation}\label{eq:cutoff-profile}
 R_\lambda(t):=\bigl(\chi\widetilde R_\lambda(t),0\bigr).
\end{equation}
It and its first derivative vanish near the vertex, so
\(R_\lambda(t)\in D(H_\alpha)\) independently of \(\alpha\).

\begin{lemma}[Profile and error estimates]\label[lemma]{lem:profile-error}
There exist \(c,\delta>0\), independent of sufficiently large \(\lambda\),
such that, for every integer \(k\ge0\),
\begin{equation}\label{eq:profile-derivatives}
 \abs{\partial_x^k\widetilde R_\lambda(t,x)}
 \le C_{k,\kappa}\mu(t)^{-1/2-k}
 \exp\!\left(-c\frac{\abs{x-a}}{\mu(t)}\right).
\end{equation}
In particular,
\begin{equation}\label{eq:profile-Linfty}
 \norm{R_\lambda(t)}_{L^\infty(\G)}\le C\mu(t)^{-1/2},\qquad
 \norm{\partial_xR_\lambda(t)}_{L^\infty(\G)}
 \le C\mu(t)^{-3/2}.
\end{equation}
Moreover,
\[
 R_\lambda\in C([0,T);D(H_\alpha))
 \cap C^1([0,T);L^2(\G)).
\]
The error
\begin{equation}\label{eq:error-definition}
 E_\lambda:=i\partial_tR_\lambda-H_\alpha R_\lambda+F(R_\lambda)
\end{equation}
belongs to \(C([0,T);\E(\G))\cap C^1([0,T);L^2(\G))\) and
\begin{equation}\label{eq:error-H1}
 \norm{E_\lambda(t)}_{H^1(\G)}
 \le C\exp\!\left(-\frac{\delta}{\mu(t)}\right).
\end{equation}
For each fixed \(\lambda\), there are integers \(M_{m,r}\ge0\) and
constants \(C_{m,r,\lambda}\) such that
\begin{equation}\label{eq:mixed-profile-estimates}
 \abs{\partial_t^r\partial_x^m\widetilde R_\lambda(t,x)}
 \le C_{m,r,\lambda}\mu(t)^{-M_{m,r}}
 e^{-c\abs{x-a}/\mu(t)},
\end{equation}
There are \(M_0,M\ge0\) such that
\begin{align}
 \norm{\partial_tR_\lambda(t)}_2
 &\le C_\lambda\mu(t)^{-M_0},\label{eq:profile-time}\\
 \norm{E_\lambda(t)}_{H^1}
 +\norm{\partial_tE_\lambda(t)}_2
 &\le C_\lambda\mu(t)^{-M}e^{-\delta/\mu(t)}.
 \label{eq:error-time}
\end{align}
\end{lemma}

\begin{proof}
Writing \(y=(x-a)/\mu(t)\), direct differentiation gives
\[
 \abs{\partial_x^k\widetilde R_\lambda(t,x)}
 \le C_k\sum_{j=0}^k
 \lambda^j\mu(t)^{-1/2-k+j}
 e^{-c\abs{x-a}/\mu(t)}.
\]
Since \(\lambda\mu(t)\le\kappa\), this implies
\eqref{eq:profile-derivatives} and \eqref{eq:profile-Linfty}.
Repeated time and space differentiation proves
\eqref{eq:mixed-profile-estimates}; the explicit formula and the fixed
cut-off also give the stated continuity properties.

On the first edge, the free equation gives
\begin{equation}\label{eq:error-explicit}
 E_{\lambda,1}
 =2\chi'\partial_x\widetilde R_\lambda
 +\chi''\widetilde R_\lambda
 +(\chi^5-\chi)\abs{\widetilde R_\lambda}^4\widetilde R_\lambda,
 \qquad E_{\lambda,2}=0.
\end{equation}
The supports of \(\chi'\), \(\chi''\) and \(\chi^5-\chi\) have a fixed
positive distance \(d_0\) from \(a\).  Every term in
\eqref{eq:error-explicit}, its first spatial derivative and its first time
derivative is therefore bounded in the required norm by a polynomial in
\(\mu^{-1}\) times \(e^{-cd_0/\mu}\).  The inequality
\(\mu^{-m}e^{-d/\mu}\le C_{m,d}e^{-d/(2\mu)}\) gives
\eqref{eq:error-H1}--\eqref{eq:error-time}, after choosing a common
smaller \(\delta>0\).  On every compact time interval $[0,T']\subset[0,T)$,
\eqref{eq:error-explicit} and the mixed derivative estimates
\eqref{eq:mixed-profile-estimates} justify differentiation
with respect to time in $L^2(\G)$ and continuity in $H^1(\G)$
by dominated convergence. Therefore,
\[
 E_\lambda\in C([0,T);H^1(\G))
 \cap C^1([0,T);L^2(\G)).
\]
Since $E_\lambda$ vanishes near the vertex, it belongs to
$C([0,T);\E(\G))$.
\end{proof}
\section{Construction of the blow-up solution}\label{sec:construction}

We first isolate the weighted final-time argument.  The different exponents
in the \(L^2\) and \(H^1\) weights absorb the derivative of the concentrating
profile.

\begin{lemma}[Weighted terminal perturbation]\label[lemma]{lem:terminal-perturbation}
Let \(T=\kappa/\lambda^2\), \(\mu(t)=\lambda(T-t)\) and let
\(R\in C([0,T);D(H_\alpha))\cap C^1([0,T);L^2(\G))\).  Suppose that
\[
 E_R:=i\partial_tR-H_\alpha R+F(R)\in C([0,T);\E(\G))
\]
and that, for some \(A,\delta>0\),
\begin{align}
 \norm{R(t)}_\infty&\le A\mu(t)^{-1/2},
 &\norm{\partial_xR(t)}_\infty&\le A\mu(t)^{-3/2},
 \label{eq:abstract-profile-bounds}\\
 \norm{E_R(t)}_{H^1(\G)}&\le Ae^{-\delta/\mu(t)}.
 \label{eq:abstract-error-bound}
\end{align}
Set
\begin{align}
 \norm{v}_{X_T}:=\sup_{0\le t<T}\bigl(&
 e^{\delta/(2\mu(t))}\norm{v(t)}_{L^2(\G)}
 +e^{\delta/(3\mu(t))}\norm{v(t)}_{H^1(\G)}\bigr),
 \label{eq:XT-norm}\\
 X_T&:=\{v\in C([0,T);\E(\G)):\norm{v}_{X_T}<\infty\}.
\end{align}
There exist \(C_0>0\) and \(\lambda_0\ge1\), depending only on \(A, \delta, \kappa\) 
and \(\alpha\), such that, for \(\lambda\ge\lambda_0\),
\begin{equation}\label{eq:abstract-Phi}
 \Phi_R(v)(t):=-i\int_t^T S_\alpha(t-s)
 \bigl(F(R+v)-F(R)+E_R\bigr)(s)\dd s
\end{equation}
has a unique fixed point in
\[
 B_{\rho_\lambda}:=\{v\in X_T:\norm{v}_{X_T}\le\rho_\lambda\},
 \qquad \rho_\lambda:=\frac{2C_0}{\lambda}.
\]
For this fixed point, \(u=R+v\) is a mild solution of \eqref{eq:nls} and
\(v(t)\to0\) in \(\E(\G)\) as \(t\uparrow T\).
\end{lemma}

\begin{proof}
Put \(\beta=\delta/2\) and \(\gamma=\delta/3\).  If \(m\ge0\),
\(0<c<d\) and \(0<r\le\kappa\), then
\[
 e^{c/r}\int_0^r\rho^{-m}e^{-d/\rho}\dd\rho
 \le\int_0^\kappa\rho^{-m}e^{-(d-c)/\rho}\dd\rho<\infty.
\]
In the borderline case \(m=2\),
\[
 e^{c/r}\int_0^r\rho^{-2}e^{-c/\rho}\dd\rho=\frac1c.
\]
Since \(\rho=\mu(s)=\lambda(T-s)\) and
\(\dd s=-\dd\rho/\lambda\), it follows that
\begin{equation}\label{eq:weighted-integral}
 \sup_{0\le t<T}e^{c/\mu(t)}
 \int_t^T\mu(s)^{-m}e^{-d/\mu(s)}\dd s
 \le\frac{C}{\lambda}
\end{equation}
whenever \(d>c\), or when \(d=c\) and \(m=2\).

Let \(X=\norm{v}_{X_T}\) and
\(\Ncal_R(v)=F(R+v)-F(R)\).  By \eqref{eq:XT-norm} and
\eqref{eq:sobolev-graph},
\begin{equation}\label{eq:v-weighted-pointwise}
 \norm{v(s)}_2\le Xe^{-\beta/\mu(s)},\qquad
 \norm{v(s)}_{H^1}+\norm{v(s)}_\infty
 \le CXe^{-\gamma/\mu(s)}.
\end{equation}
Using \eqref{eq:F-Lipschitz},
\begin{equation}\label{eq:nonlinear-L2}
 \norm{\Ncal_R(v)(s)}_2
 \le C_A\mu(s)^{-2}Xe^{-\beta/\mu(s)}
 +CX^5e^{-(4\gamma+\beta)/\mu(s)}.
\end{equation}
On each edge the chain rule gives
\[
 \partial_x\Ncal_R(v)
 =DF(R+v)\partial_xv+[DF(R+v)-DF(R)]\partial_xR.
\]
Equations \eqref{eq:abstract-profile-bounds},
\eqref{eq:DF-Lipschitz} and \eqref{eq:v-weighted-pointwise} yield
\begin{align}
 \norm{\Ncal_R(v)(s)}_{H^1}
 \le C_A\bigl(&\mu(s)^{-2}Xe^{-\gamma/\mu(s)}
 +\mu(s)^{-3}Xe^{-\beta/\mu(s)}\notag\\
 &+\mu(s)^{-3/2}X^4e^{-(3\gamma+\beta)/\mu(s)}
 +X^5e^{-5\gamma/\mu(s)}\bigr).
 \label{eq:nonlinear-H1}
\end{align}
The four terms arise from
\(\abs{R}^4\abs{\partial_xv}\),
\(\abs{R}^3\abs{v}\abs{\partial_xR}\),
\(\abs{v}^4\abs{\partial_xR}\) and
\(\abs{v}^4\abs{\partial_xv}\), respectively.

Unitarity on \(L^2\), \eqref{eq:group-H1},
\eqref{eq:abstract-error-bound} and \eqref{eq:weighted-integral} give
\begin{align}
 \sup_{t<T}e^{\beta/\mu(t)}\norm{\Phi_R(v)(t)}_2
 &\le\frac{C_A}{\lambda}(1+X+X^5),\label{eq:Phi-L2}\\
 \sup_{t<T}e^{\gamma/\mu(t)}\norm{\Phi_R(v)(t)}_{H^1}
 &\le\frac{C_A}{\lambda}(1+X+X^4+X^5).
 \label{eq:Phi-H1}
\end{align}
The only borderline integrals are those with
\(\mu^{-2}e^{-\beta/\mu}\) at the \(L^2\) level and
\(\mu^{-2}e^{-\gamma/\mu}\) at the \(H^1\) level.  The term with
\(\mu^{-3}e^{-\beta/\mu}\) has the strict gap \(\beta>\gamma\).
Thus
\begin{equation}\label{eq:Phi-self}
 \norm{\Phi_R(v)}_{X_T}
 \le\frac{C_0}{\lambda}
 +\frac{C_1}{\lambda}(X+X^4+X^5).
\end{equation}

Let \(v,w\in X_T\), \(h=v-w\), \(Y=\norm{h}_{X_T}\) and
\(\norm{v}_{X_T},\norm{w}_{X_T}\le r\).  The same calculation gives
\begin{align*}
 \norm{\Ncal_R(v)-\Ncal_R(w)}_2
 \le{}&C_A\mu^{-2}Ye^{-\beta/\mu}
 +Cr^4Ye^{-(4\gamma+\beta)/\mu},\\
 \norm{\Ncal_R(v)-\Ncal_R(w)}_{H^1}
 \le C_AY\bigl(&
 \mu^{-2}e^{-\gamma/\mu}
 +\mu^{-3}e^{-\beta/\mu}\\
 &+r\mu^{-3/2}e^{-2\gamma/\mu}\\
 &+r^3\mu^{-3/2}e^{-(3\gamma+\beta)/\mu}\\
 &+r^4e^{-5\gamma/\mu}\bigr),
\end{align*}
where all functions are evaluated at the same time.  Therefore
\begin{equation}\label{eq:Phi-contraction}
 \norm{\Phi_R(v)-\Phi_R(w)}_{X_T}
 \le\frac{C_2}{\lambda}(1+r)^4\norm{v-w}_{X_T}.
\end{equation}

Set \(\rho_\lambda=2C_0/\lambda\).  For \(\lambda\) large,
\eqref{eq:Phi-self} maps \(B_{\rho_\lambda}\) into itself and
\eqref{eq:Phi-contraction} has contraction constant at most \(1/2\).
Banach's theorem gives the fixed point and the estimates encoded in
\eqref{eq:XT-norm}.  For \(0\le t_0\le t<T\), the fixed-point identity is
\[
 v(t)=S_\alpha(t-t_0)v(t_0)
 +i\int_{t_0}^tS_\alpha(t-s)
 \bigl(F(R+v)-F(R)+E_R\bigr)(s)\dd s.
\]
The definition of \(E_R\) gives the analogous identity for \(R\).
Adding them yields \eqref{eq:duhamel-solution} for \(u=R+v\).
\end{proof}

\begin{proposition}[Energy space construction]\label[proposition]{prop:energy-construction}
For \(\lambda\) sufficiently large, the profile \(R_\lambda\) from
\eqref{eq:cutoff-profile} admits a correction \(v_\lambda\in X_T\) such
that \(u_\lambda=R_\lambda+v_\lambda\) is a mild solution of
\eqref{eq:nls} and
\begin{equation}\label{eq:v-bounds}
 \norm{v_\lambda(t)}_2
 \le\rho_\lambda e^{-\delta/(2\mu(t))},\quad
 \norm{v_\lambda(t)}_{H^1}
 \le\rho_\lambda e^{-\delta/(3\mu(t))},\quad
 \rho_\lambda=\frac{2C_0}{\lambda}.
\end{equation}
\end{proposition}

\begin{proof}
The bounds in \cref{lem:profile-error} verify the hypotheses of
\cref{lem:terminal-perturbation}; its fixed point gives the result.
By energy-class uniqueness, this mild solution is the restriction of the
maximal energy solution, so the conservation laws in
\cref{prop:framework} apply to it.
\end{proof}

\begin{proposition}[Localization and blow-up rate]\label[proposition]{prop:localization}
The solution in \cref{prop:energy-construction} satisfies
\eqref{eq:local-gradient-asymptotic}, \eqref{eq:active-away} and
\eqref{eq:inactive-edge}.
\end{proposition}

\begin{proof}
Since \(R_{\lambda,2}=0\), we have \(u_{\lambda,2}=v_{\lambda,2}\) and
\eqref{eq:inactive-edge} follows from \eqref{eq:v-bounds}.

Fix \(0<\eta<a/2\).  On \(I_\eta=(a-\eta,a+\eta)\), \(\chi=1\).
With \(y=(x-a)/\mu(t)\), differentiation of \eqref{eq:profile} gives
\begin{equation}\label{eq:profile-derivative-formula}
 \partial_x\widetilde R_\lambda(t,x)
 =\mu(t)^{-3/2}
 \left(Q'(y)-\frac{i}{2}\lambda\mu(t)yQ(y)\right)
 e^{i\theta_\lambda(t,x)}.
\end{equation}
Since \(Q,Q'\) are real,
\begin{align}
 \mu(t)^2\norm{\partial_x\widetilde R_\lambda(t)}_{L^2(I_\eta)}^2
 ={}&\int_{-\eta/\mu(t)}^{\eta/\mu(t)}\abs{Q'(y)}^2\dd y\notag\\
 &+\frac{\lambda^2\mu(t)^2}{4}
 \int_{-\eta/\mu(t)}^{\eta/\mu(t)}y^2Q(y)^2\dd y.
 \label{eq:profile-gradient-limit}
\end{align}
The first term tends to \(\norm{Q'}_2^2\) and the second tends to zero.
The bound \eqref{eq:v-bounds} also gives
\[
 \mu(t)\norm{\partial_xv_{\lambda,1}(t)}_2\longrightarrow0.
\]
The reverse triangle inequality proves
\eqref{eq:local-gradient-asymptotic}.

Let \(K\Subset[0,\infty)\setminus\{a\}\) and
\(d_K=\dist(K,\{a\})>0\).  From \eqref{eq:profile-derivatives},
\begin{equation}\label{eq:R-away}
 \norm{R_{\lambda,1}(t)}_{H^1(K)}
 \le C_K\bigl(\mu(t)^{-1/2}+\mu(t)^{-3/2}\bigr)
 e^{-cd_K/\mu(t)}\longrightarrow0.
\end{equation}
The same conclusion for \(v_{\lambda,1}\) follows from
\eqref{eq:v-bounds}, proving \eqref{eq:active-away}.  In particular, the
local divergence at \(a\) and the local bounds away from \(a\) show that the
blow-up set is precisely the point \(a\) on edge \(1\).  The divergence of
the full energy norm and the blow-up alternative show that the solution
cannot be continued through \(T\).
\end{proof}

\begin{proposition}[Mass and concentration]\label[proposition]{prop:mass-concentration}
For the constructed solution,
\[
 M(u_\lambda(t))=M_{\R}(Q)\qquad(0\le t<T),
\]
and \eqref{eq:mass-concentration} holds.
\end{proposition}

\begin{proof}
Extend \(\chi\) by zero to \(\R\).  With \(y=(x-a)/\mu(t)\),
\begin{equation}\label{eq:R-mass-limit}
 M(R_\lambda(t))
 =\frac12\int_\R\abs{\chi(a+\mu(t)y)}^2Q(y)^2\dd y
 \longrightarrow M_{\R}(Q)
\end{equation}
by dominated convergence.  Since \(v_\lambda(t)\to0\) in \(L^2\),
\(M(u_\lambda(t))\to M_{\R}(Q)\).  Mass conservation makes
\(M(u_\lambda(t))\) constant, proving the identity for all \(t<T\).

Fix \(\eps>0\) and choose \(0<\eta<\min\{\eps,a/2\}\).  On \(I_\eta\),
\[
 \int_{\abs{x-a}\le\eta}\abs{R_{\lambda,1}(t,x)}^2\dd x
 =\int_{-\eta/\mu(t)}^{\eta/\mu(t)}Q(y)^2\dd y
 \longrightarrow\norm Q_2^2.
\]
The difference between this integral and the one for
\(u_{\lambda,1}=R_{\lambda,1}+v_{\lambda,1}\) is bounded by
\[
 2\norm{R_{\lambda,1}(t)}_{L^2(I_\eta)}
 \norm{v_{\lambda,1}(t)}_2+\norm{v_{\lambda,1}(t)}_2^2,
\]
which tends to zero. Moreover,
\[
 \int_{I_\eta}|u_{\lambda,1}(t,x)|^2\dd x
 \le
 \int_{\{x\ge0:\,|x-a|\le\eps\}}
 |u_{\lambda,1}(t,x)|^2\dd x
 \le \norm{u_\lambda(t)}_{L^2(\G)}^2
 =2M_{\R}(Q)=\norm{Q}_{L^2(\R)}^2.
\]
Since the first integral converges to
$\norm{Q}_{L^2(\R)}^2$, squeezing proves
\eqref{eq:mass-concentration}.
\end{proof}

\section{Persistence of the operator domain}\label{sec:regularity}

The solution constructed above is initially known only in the energy space.
We now prove that it is a strong solution.  Constants in this section may
depend on the fixed \(\lambda\), but all estimates used in the limiting
argument are uniform in the approximation index and in the difference
quotient.

\begin{lemma}[Uniform final-time contraction]\label[lemma]{lem:finite-time}
For \(0<\tau<T\), define
\[
 \Phi_\tau(v)(t):=-i\int_t^\tau S_\alpha(t-s)
 \bigl(F(R_\lambda+v)-F(R_\lambda)+E_\lambda\bigr)(s)\dd s,
 \qquad 0\le t\le\tau,
\]
and equip \(C([0,\tau];\E(\G))\) with the norm
\[
 \norm v_{X_\tau}:=\sup_{0\le t\le\tau}
 \left(e^{\delta/(2\mu(t))}\norm{v(t)}_2
 +e^{\delta/(3\mu(t))}\norm{v(t)}_{H^1}\right).
\]
For the same sufficiently large \(\lambda\) as above, the ball radius
\(\rho_\lambda\) and a contraction constant \(q_\lambda<1\) can be chosen
independently of \(\tau\).  Consequently \(\Phi_\tau\) has a unique fixed
point \(v_{\lambda,\tau}\) in that ball.
\end{lemma}

\begin{proof}
The proof of \cref{lem:terminal-perturbation} is unchanged with \(\tau\) in
place of \(T\).  Indeed, every nonnegative integral is only shortened and
\[
 \int_t^\tau\mu(s)^{-m}e^{-c/\mu(s)}\dd s
 =\frac1\lambda\int_{\mu(\tau)}^{\mu(t)}
 r^{-m}e^{-c/r}\dd r
 \le\frac1\lambda\int_0^{\mu(t)}r^{-m}e^{-c/r}\dd r.
\]
Thus \eqref{eq:Phi-self}--\eqref{eq:Phi-contraction} hold with constants
independent of \(\tau\).
\end{proof}

\begin{proposition}[Strong terminal approximants]\label[proposition]{prop:approximants}
Let \(t_n\uparrow T\).  For each \(n\), let \(u_n\) be the maximal strong
solution of \eqref{eq:nls} with terminal datum
\[
 u_n(t_n)=R_\lambda(t_n).
\]
Its maximal interval contains \([0,t_n]\) and
\[
 v_{\lambda,n}:=u_n-R_\lambda
\]
is the fixed point of \(\Phi_{t_n}\).  In particular,
\begin{equation}\label{eq:approximant-v-bounds}
 \norm{v_{\lambda,n}(t)}_2
 \le\rho_\lambda e^{-\delta/(2\mu(t))},\qquad
 \norm{v_{\lambda,n}(t)}_{H^1}
 \le\rho_\lambda e^{-\delta/(3\mu(t))}
\end{equation}
on \([0,t_n]\), uniformly in \(n\).
\end{proposition}

\begin{proof}
Since \(R_\lambda(t_n)\in D(H_\alpha)\), the strong Cauchy theory in
\cref{prop:framework} gives a unique maximal strong solution through
\(t_n\).  On the other hand, \cref{lem:finite-time} produces a mild
solution
\[
 \widetilde u_n=R_\lambda+v_{\lambda,t_n}
 \quad\text{on }[0,t_n],\qquad
 \widetilde u_n(t_n)=R_\lambda(t_n).
\]
Energy-class uniqueness identifies the two solutions on their common
interval.  If the maximal strong interval had a finite left endpoint in
\([0,t_n)\), the blow-up alternative would force the energy norm to diverge
there.  This contradicts the fixed-point bound on \(\widetilde u_n\).
Hence the interval contains \([0,t_n]\) and uniqueness of the fixed point
gives \(v_{\lambda,n}=v_{\lambda,t_n}\), together with
\eqref{eq:approximant-v-bounds}.
\end{proof}

\begin{proposition}[Convergence of the approximants]\label[proposition]{prop:approximant-convergence}
For every \(T'<T\),
\begin{equation}\label{eq:approximant-convergence}
 u_n\longrightarrow u_\lambda
 \qquad\text{in }C([0,T'];L^2(\G)).
\end{equation}
\end{proposition}

\begin{proof}
For \(n\) large, \(T'<t_n\).  Set
\[
 d_n=u_\lambda-u_n=v_\lambda-v_{\lambda,n},\qquad
 \mu_n=\mu(t_n),\qquad \beta=\frac{\delta}{2}.
\]
Since $u_\lambda$ and $u_n$ are mild solutions,
\eqref{eq:duhamel-solution} gives, for $0\le t\le t_n$,
\[
 d_n(t)=S_\alpha(t-t_n)d_n(t_n)
 -i\int_t^{t_n}S_\alpha(t-s)
 \bigl(F(u_\lambda(s))-F(u_n(s))\bigr)\dd s.
\]
The bounds \eqref{eq:profile-Linfty}, \eqref{eq:v-bounds}
and \eqref{eq:approximant-v-bounds} imply
\[
 \norm{u_\lambda(t)}_\infty+\norm{u_n(t)}_\infty
 \le C\mu(t)^{-1/2},
\]
where $C$ is independent of $n$ and of sufficiently large
$\lambda$.
Hence
\[
 \norm{F(u_\lambda(t))-F(u_n(t))}_2
 \le C\mu(t)^{-2}\norm{d_n(t)}_2.
\]
Backward Duhamel and Gronwall give
\begin{align*}
 \norm{d_n(t)}_2
 &\le\norm{d_n(t_n)}_2
 \exp\!\left(C\int_t^{t_n}\mu(s)^{-2}\dd s\right)\\
 &\le\rho_\lambda
 \exp\!\left(
 -\frac{\beta-C/\lambda}{\mu_n}
 -\frac{C}{\lambda\mu(t)}\right).
\end{align*}
Here \(d_n(t_n)=v_\lambda(t_n)\) and
\[
 \int_t^{t_n}\mu(s)^{-2}\dd s
 =\frac1\lambda\left(\frac1{\mu_n}-\frac1{\mu(t)}\right).
\]
Increasing \(\lambda_0\) so that \(\beta-C/\lambda>0\), the last bound
tends to zero uniformly on \([0,T']\), proving
\eqref{eq:approximant-convergence}.
\end{proof}

\begin{proposition}[Difference-quotient reduction]\label[proposition]{prop:difference-equation}
For \(0<h<t_n\), set
\[
 z_{n,h}(t):=\frac{v_{\lambda,n}(t+h)-v_{\lambda,n}(t)}{h},
 \qquad
 r_{\lambda,h}(t):=\frac{R_\lambda(t+h)-R_\lambda(t)}{h}.
\]
Then
\[
 z_{n,h}\in C([0,t_n-h];D(H_\alpha))
 \cap C^1([0,t_n-h];L^2(\G)).
\]
For \(0\le t\le t_n-h\) and \(0\le\theta\le1\), define
\[
 U_{n,h,\theta}(t)
 :=u_n(t)+\theta\bigl(u_n(t+h)-u_n(t)\bigr),
\]
\[
 V_{h,\theta}(t)
 :=R_\lambda(t)+\theta\bigl(R_\lambda(t+h)-R_\lambda(t)\bigr),
\]
and
\begin{align*}
 A_{n,h}&:=\int_0^1 3\abs{U_{n,h,\theta}}^4\dd\theta,&
 B_{n,h}&:=\int_0^1
 2\abs{U_{n,h,\theta}}^2U_{n,h,\theta}^2\dd\theta,\\
 A_{\lambda,h}&:=\int_0^1 3\abs{V_{h,\theta}}^4\dd\theta,&
 B_{\lambda,h}&:=\int_0^1
 2\abs{V_{h,\theta}}^2V_{h,\theta}^2\dd\theta.
\end{align*}
Then
\begin{equation}\label{eq:difference-equation}
 i\partial_tz_{n,h}-H_\alpha z_{n,h}
 +A_{n,h}z_{n,h}+B_{n,h}\overline{z_{n,h}}=f_{n,h},
\end{equation}
where
\begin{align}
 f_{n,h}={}&-(A_{n,h}-A_{\lambda,h})r_{\lambda,h}
 -(B_{n,h}-B_{\lambda,h})\overline{r_{\lambda,h}}
 -\delta_hE_\lambda,\label{eq:difference-force}\\
 \delta_hE_\lambda(t):={}&
 \frac{E_\lambda(t+h)-E_\lambda(t)}h.\notag
\end{align}
Moreover,
\begin{equation}\label{eq:difference-terminal}
 z_{n,h}(t_n-h)=-\frac{v_{\lambda,n}(t_n-h)}h,
\end{equation}
and, for every \(T'<t_n\),
\begin{equation}\label{eq:difference-limit}
 z_{n,h}\longrightarrow\partial_tv_{\lambda,n}
 \quad\text{in }C([0,T'];L^2(\G))
 \quad(h\downarrow0).
\end{equation}
\end{proposition}

\begin{proof}
The remainder satisfies
\begin{equation}\label{eq:remainder-strong}
 i\partial_tv_{\lambda,n}-H_\alpha v_{\lambda,n}
 +F(R_\lambda+v_{\lambda,n})-F(R_\lambda)+E_\lambda=0.
\end{equation}
For \(F(z)=\abs z^4z\), its real differential is
\begin{equation}\label{eq:F-differential}
 DF(z)\eta=3\abs z^4\eta+2\abs z^2z^2\overline\eta.
\end{equation}
The fundamental theorem of calculus along the segment between
\(u_n(t)\) and \(u_n(t+h)\) gives
\[
 \delta_hF(u_n)
 =A_{n,h}\delta_hu_n+B_{n,h}\overline{\delta_hu_n}.
\]
The same identity for \(R_\lambda\), together with
\(\delta_hu_n=z_{n,h}+r_{\lambda,h}\), gives
\eqref{eq:difference-equation}--\eqref{eq:difference-force}.
Since \(v_{\lambda,n}(t_n)=0\), \eqref{eq:difference-terminal} follows.

Finally, \(v_{\lambda,n}\in C^1([0,t_n];L^2)\).  If
\(T'+h_0<t_n\), then
\[
 z_{n,h}(t)-\partial_tv_{\lambda,n}(t)
 =\frac1h\int_0^h
 \bigl(\partial_tv_{\lambda,n}(t+s)
 -\partial_tv_{\lambda,n}(t)\bigr)\dd s.
\]
Uniform continuity of \(\partial_tv_{\lambda,n}\) on
\([0,T'+h_0]\) proves \eqref{eq:difference-limit}.
\end{proof}

\begin{proposition}[Uniform time-derivative bound]\label[proposition]{prop:time-derivative}
For every \(T'<T\), there are \(n_0\) and \(C_{T'}>0\) such that
\begin{equation}\label{eq:time-derivative-bound}
 \sup_{n\ge n_0}\sup_{0\le t\le T'}
 \norm{\partial_tv_{\lambda,n}(t)}_{L^2(\G)}
 \le C_{T'}.
\end{equation}
\end{proposition}

\begin{proof}
Choose \(h_0>0\) with \(T'+h_0<T\) and then \(n_0\) so that
\(T'+h_0<t_n\) for \(n\ge n_0\).  For a fixed such \(n\), restrict
\begin{equation}\label{eq:h-range}
 0<h\le\overline h_n:=
 \min\left\{h_0,\frac{T-t_n}{2}\right\}.
\end{equation}
All estimates below are uniform in \(n\ge n_0\) and in this range of \(h\).

\smallskip
Let
\[
 a_0(z)=3\abs z^4,\qquad b_0(z)=2\abs z^2z^2.
\]
The mean-value theorem on \(\R^2\) gives
\begin{equation}\label{eq:ab-lipschitz}
 \abs{a_0(z)-a_0(w)}+\abs{b_0(z)-b_0(w)}
 \le C(\abs z^3+\abs w^3)\abs{z-w}.
\end{equation}
By \eqref{eq:profile-Linfty} and
\eqref{eq:approximant-v-bounds},
\begin{equation}\label{eq:un-Linfty}
 \norm{u_n(t)}_\infty\le C\mu(t)^{-1/2},
 \qquad 0\le t\le t_n.
\end{equation}
As \(\mu\) is decreasing,
\[
 \norm{U_{n,h,\theta}(t)}_\infty
 +\norm{V_{h,\theta}(t)}_\infty
 \le C\mu(t+h)^{-1/2}.
\]
Moreover,
\[
 \norm{U_{n,h,\theta}(t)-V_{h,\theta}(t)}_\infty
 \le Ce^{-\delta/(3\mu(t))}.
\]
Thus \eqref{eq:ab-lipschitz} yields
\begin{align}
 &\norm{A_{n,h}(t)-A_{\lambda,h}(t)}_\infty
 +\norm{B_{n,h}(t)-B_{\lambda,h}(t)}_\infty\notag\\
 &\hspace{35mm}\le
 C\mu(t+h)^{-3/2}e^{-\delta/(3\mu(t))},
 \label{eq:coefficient-difference}\\
 &\norm{B_{n,h}(t)}_\infty\le C\mu(t+h)^{-2}.
 \label{eq:B-bound}
\end{align}
The constant in \eqref{eq:B-bound} is uniform for \(\lambda\) in the
perturbative range.

The fundamental theorem of calculus and
\eqref{eq:profile-time} imply
\begin{equation}\label{eq:r-bound}
 \norm{r_{\lambda,h}(t)}_2
 \le\sup_{s\in[t,t+h]}\norm{\partial_tR_\lambda(s)}_2
 \le C_\lambda\mu(t+h)^{-M_0}.
\end{equation}
From \eqref{eq:h-range}, for \(s\in[t,t+h]\),
\[
 \mu(t+h)\le\mu(s)\le\mu(t)\le\frac32\mu(t+h).
\]
Using \eqref{eq:error-time}, we therefore obtain
\begin{equation}\label{eq:delta-E-bound}
 \norm{\delta_hE_\lambda(t)}_2
 \le C_\lambda\mu(t+h)^{-M}
 e^{-\delta/(2\mu(t+h))}.
\end{equation}
Combining \eqref{eq:difference-force},
\eqref{eq:coefficient-difference}, \eqref{eq:r-bound} and
\eqref{eq:delta-E-bound}, we find an integer \(M_1\ge0\) such that
\begin{equation}\label{eq:force-bound}
 \norm{f_{n,h}(t)}_2
 \le C_\lambda\mu(t+h)^{-M_1}
 e^{-\delta/(3\mu(t))}.
\end{equation}

\smallskip
Put
\[
 \gamma:=\frac{\delta}{12},\qquad
 \omega_h(t):=e^{2\gamma/\mu(t+h)}.
\]
Since \(\mu'(t)=-\lambda\),
\begin{equation}\label{eq:weight-derivative}
 \frac{\omega_h'(t)}{\omega_h(t)}
 =\frac{2\gamma\lambda}{\mu(t+h)^2}.
\end{equation}
Take the \(L^2(\G)\) pairing of
\eqref{eq:difference-equation} with
\(\omega_h\overline{z_{n,h}}\) and then its imaginary part.  Self-adjointness
of \(H_\alpha\) removes its contribution and \(A_{n,h}\) is real.
Consequently,
\begin{align}
 \frac12\frac{\dd}{\dd t}
 \bigl(\omega_h\norm{z_{n,h}}_2^2\bigr)
 ={}&\frac12\omega_h'\norm{z_{n,h}}_2^2
 -\omega_h\operatorname{Im}
 \int_{\G}B_{n,h}\overline{z_{n,h}}^{\,2}\dd x\notag\\
 &+\omega_h\operatorname{Im}
 \int_{\G}f_{n,h}\overline{z_{n,h}}\dd x.
 \label{eq:weighted-energy-identity}
\end{align}
By \eqref{eq:B-bound},
\[
 \left|\operatorname{Im}\int_{\G}
 B_{n,h}\overline{z_{n,h}}^{\,2}\dd x\right|
 \le C\mu(t+h)^{-2}\norm{z_{n,h}}_2^2.
\]
Choose \(\lambda\) still larger if necessary so that
\(\gamma\lambda\ge2C\).  Using Young's inequality for the forcing term in
\eqref{eq:weighted-energy-identity}, we obtain
\[
 -\frac{\dd}{\dd t}
 \bigl(\omega_h\norm{z_{n,h}}_2^2\bigr)
 \le C\mu(t+h)^2\omega_h\norm{f_{n,h}}_2^2.
\]
Integration from \(t\) to \(t_n-h\) gives
\begin{align}
 \omega_h(t)\norm{z_{n,h}(t)}_2^2
 \le{}&\omega_h(t_n-h)\norm{z_{n,h}(t_n-h)}_2^2\notag\\
 &+C\int_t^{t_n-h}\mu(s+h)^2\omega_h(s)
 \norm{f_{n,h}(s)}_2^2\dd s.
 \label{eq:weighted-energy-integrated}
\end{align}

\smallskip
The final-time fixed-point identity is
\[
 v_{\lambda,n}(t)
 =-i\int_t^{t_n}S_\alpha(t-s)G_n(s)\dd s,
\]
where
\[
 G_n:=F(R_\lambda+v_{\lambda,n})-F(R_\lambda)+E_\lambda.
\]
The same nonlinear estimates as in
\cref{lem:terminal-perturbation} give, for some \(M_2\ge0\),
\begin{equation}\label{eq:Gn-bound}
 \norm{G_n(t)}_2\le
 C_\lambda\mu(t)^{-M_2}e^{-\delta/(3\mu(t))}.
\end{equation}
Put \(\mu_n=\mu(t_n)\).  On \([t_n-h,t_n]\),
\[
 \mu_n\le\mu(s)\le\mu_n+\lambda h\le\frac32\mu_n.
\]
Using \eqref{eq:difference-terminal}, unitarity and
\eqref{eq:Gn-bound},
\begin{align}
 \norm{z_{n,h}(t_n-h)}_2
 &\le\frac1h\int_{t_n-h}^{t_n}\norm{G_n(s)}_2\dd s\notag\\
 &\le C_\lambda\mu_n^{-M_2}
 e^{-2\delta/(9\mu_n)}.
 \label{eq:terminal-z-bound}
\end{align}
Since \(\omega_h(t_n-h)=e^{2\gamma/\mu_n}\) and
\(\gamma=\delta/12\),
\begin{equation}\label{eq:terminal-weighted-bound}
 \omega_h(t_n-h)\norm{z_{n,h}(t_n-h)}_2^2
 \le C_\lambda\mu_n^{-2M_2}e^{-5\delta/(18\mu_n)}
 \le C_\lambda.
\end{equation}

\smallskip
For \(0\le s\le t_n-h\), \eqref{eq:h-range} implies
\[
 \mu(s+h)\le\mu(s)\le\frac32\mu(s+h).
\]
Therefore \eqref{eq:force-bound} and the choice
\(\gamma=\delta/12\) give, for some \(M_3\ge0\),
\begin{align}
 \mu(s+h)^2\omega_h(s)\norm{f_{n,h}(s)}_2^2
 &\le C_\lambda\mu(s+h)^{2-2M_1}
 \exp\!\left(\frac{2\gamma}{\mu(s+h)}
 -\frac{2\delta}{3\mu(s)}\right)\notag\\
 &\le C_\lambda\mu(s)^{-M_3}
 e^{-5\delta/(12\mu(s))}.
 \label{eq:weighted-force-integrand}
\end{align}
This is integrable uniformly up to \(T\), since
\[
 \int_0^T\mu(s)^{-M_3}e^{-5\delta/(12\mu(s))}\dd s
 =\frac1\lambda\int_0^{\kappa/\lambda}
 r^{-M_3}e^{-5\delta/(12r)}\dd r<\infty.
\]
Combining this with \eqref{eq:weighted-energy-integrated} and
\eqref{eq:terminal-weighted-bound} gives
\[
 \omega_h(t)\norm{z_{n,h}(t)}_2^2\le C_{T'},
 \qquad 0\le t\le T'.
\]
Since \(\omega_h\ge1\), these quotients are uniformly bounded on
\([0,T']\) with values in \(L^2\).  For each fixed \(n\), we may pass to
the limit \(h\downarrow0\) using \eqref{eq:difference-limit}.  This proves
\eqref{eq:time-derivative-bound}.
\end{proof}

\begin{corollary}[Uniform graph-norm bounds]\label[corollary]{cor:graph-bounds}
For every \(T'<T\), there are \(n_0\) and \(C_{T'}\) such that
\begin{equation}\label{eq:H-v-bound}
 \sup_{n\ge n_0}\sup_{0\le t\le T'}
 \norm{H_\alpha v_{\lambda,n}(t)}_2\le C_{T'},
\end{equation}
and
\begin{equation}\label{eq:H-u-bound}
 \sup_{n\ge n_0}\sup_{0\le t\le T'}
 \norm{H_\alpha u_n(t)}_2\le C_{T'}.
\end{equation}
\end{corollary}

\begin{proof}
Equation \eqref{eq:remainder-strong} gives
\[
 H_\alpha v_{\lambda,n}
 =i\partial_tv_{\lambda,n}
 +F(R_\lambda+v_{\lambda,n})-F(R_\lambda)+E_\lambda.
\]
The first term is controlled by \cref{prop:time-derivative}.  On
\([0,T']\), the second is uniformly bounded in \(L^2\) by
\[
 \norm{F(R_\lambda+v_{\lambda,n})-F(R_\lambda)}_2
 \le C\bigl(\norm{R_\lambda}_\infty^4
 +\norm{v_{\lambda,n}}_\infty^4\bigr)\norm{v_{\lambda,n}}_2,
\]
and the error is bounded by \cref{lem:profile-error}.  This proves
\eqref{eq:H-v-bound}.  Since \(H_\alpha R_\lambda\) is bounded on compact
time intervals, \eqref{eq:H-u-bound} follows.  Together with the already
known \(L^2\)-bounds, these are uniform graph-norm bounds.
\end{proof}

\begin{proposition}[Domain persistence]\label[proposition]{prop:domain-persistence}
For every \(t<T\), \(u_\lambda(t)\in D(H_\alpha)\).
\end{proposition}

\begin{proof}
Fix \(T'<T\) and \(t\in[0,T']\).  By
\eqref{eq:approximant-convergence}, \(u_n(t)\to u_\lambda(t)\) strongly in
\(L^2\), while \eqref{eq:H-u-bound} permits extraction of a subsequence
such that
\[
 H_\alpha u_n(t)\rightharpoonup g_t
 \quad\text{weakly in }L^2(\G).
\]
The graph of the closed linear operator \(H_\alpha\) is a closed linear
subspace of \(L^2(\G)\times L^2(\G)\) and is therefore weakly closed.
Since
\[
 (u_n(t),H_\alpha u_n(t))\in\operatorname{Graph}(H_\alpha),
\]
we infer
\[
 (u_\lambda(t),g_t)\in\operatorname{Graph}(H_\alpha).
\]
Thus \(u_\lambda(t)\in D(H_\alpha)\).  Since \(T'<T\) and \(t\in[0,T']\) are arbitrary,
the result is proved.
\end{proof}

\begin{proposition}[Strong regularity]\label[proposition]{prop:strong-regularity}
The solution constructed in \cref{prop:energy-construction} satisfies
\[
 u_\lambda\in C([0,T);D(H_\alpha))
 \cap C^1([0,T);L^2(\G)).
\]
\end{proposition}

\begin{proof}
Fix \(t_0<T\).  By \cref{prop:domain-persistence},
\(u_\lambda(t_0)\in D(H_\alpha)\).  Strong local well-posedness produces a
solution
\[
 U\in C(J;D(H_\alpha))\cap C^1(J;L^2(\G))
\]
on a neighborhood \(J\) of \(t_0\), with
\(U(t_0)=u_\lambda(t_0)\).  Both \(U\) and \(u_\lambda\) are energy
solutions, so energy-class uniqueness shows that they agree on
\(J\cap[0,T)\).  Such neighborhoods cover \([0,T)\), proving the claim.
\end{proof}

\begin{proof}[Proof of \cref{thm:main}]
\Cref{prop:energy-construction} gives the energy solution and the correction
bounds.  \Cref{prop:localization} gives the precise blow-up rate and the
edgewise decay, while \cref{prop:mass-concentration} gives the exact mass
and concentration.  Finally, \cref{prop:strong-regularity} gives the stated
operator-domain regularity.
\end{proof}

\section{Terminal parameters and open problems}\label{sec:terminal-family}

The qualitative properties stated in \cref{thm:main} do not determine a unique 
finite-time blow-up solution. 
Indeed, this construction admits two degrees of freedom arising from the group of symmetries
acting on the pseudoconformal profile in \eqref{eq:euclidean-family}: 
the concentration scale \(\ell\) and the position of the renormalized center \(p\).

\begin{proposition}[Two-parameter terminal family]\label[proposition]{prop:terminal-family}
Fix \(\alpha\in\R\), \(a>0\), a blow-up time \(T\in\R\) and take edge \(1\) to be the blow-up edge.  
For every \(p\in\R\), there are
\(\ell_0(p)>0\) and \(\kappa>0\) such that, for every
\(\ell\ge\ell_0(p)\), equation \eqref{eq:nls} has a strong solution
\[
 U_{\ell,p}\in
 C([T-\kappa/\ell^2,T);D(H_\alpha))
 \cap C^1([T-\kappa/\ell^2,T);L^2(\G)).
\]
With \(\mu_\ell(t)=\ell(T-t)\) and the cut-off from
\eqref{eq:cutoff}, it has the decomposition
\begin{equation}\label{eq:terminal-decomposition}
 U_{\ell,p}=R_{\ell,p}+v_{\ell,p},\qquad
 R_{\ell,p}(t)=\bigl(\chi S_{a,T,\ell,p,0}(t),0\bigr),
\end{equation}
and, for constants \(C_p,\delta_p>0\),
\begin{equation}\label{eq:terminal-correction}
 \norm{v_{\ell,p}(t)}_{H^1(\G)}
 \le\frac{C_p}{\ell}
 e^{-\delta_p/(3\mu_\ell(t))}.
\end{equation}
It has mass \(M_{\R}(Q)\), blows up only at \(a\) on edge \(1\), converges to
zero locally in \(H^1\) on that edge away from \(a\) and its second
component converges to zero in \(H^1(\R_+)\).

Distinct parameter values yield distinct solutions modulo constant phase: if
\((\ell,p)\ne(\widetilde\ell,\widetilde p)\), then no
\(\vartheta\in\R\) can satisfy
\[
 U_{\ell,p}(t)=e^{i\vartheta}U_{\widetilde\ell,\widetilde p}(t)
\]
throughout a common terminal interval.
\end{proposition}

\begin{proof}
Work on \(I_\ell=[T-\kappa/\ell^2,T)\).  From
\cref{lem:euclidean-profile}, \(S_{a,T,\ell,p,0}\) is an exact free solution.
Repeated differentiation gives, for \(k\ge0\),
\begin{equation}\label{eq:terminal-profile-bound}
 \abs{\partial_x^kS_{a,T,\ell,p,0}(t,x)}
 \le C_{k,p,\kappa}\mu_\ell(t)^{-1/2-k}
 \exp\!\left(
 -c\left|\frac{x-a}{\mu_\ell(t)}-p\right|\right).
\end{equation}
The cut-off profile in \eqref{eq:terminal-decomposition} lies in
\(D(H_\alpha)\).  Its error is supported where the cut-off varies, a fixed
positive distance from \(a\).  For fixed \(p\), taking \(\ell\) large makes
\(\mu_\ell(t)\abs p\) smaller than half that distance throughout
\(I_\ell\).  Thus \eqref{eq:terminal-profile-bound} gives
\[
 \norm{i\partial_tR_{\ell,p}-H_\alpha R_{\ell,p}
 +F(R_{\ell,p})}_{H^1}
 \le C_p e^{-\delta_p/\mu_\ell(t)}.
\]
After translating the time origin, all hypotheses of
\cref{lem:terminal-perturbation} hold with \(\lambda\) replaced by
\(\ell\).  The change of variables
\(r=\mu_\ell(s)=\ell(T-s)\) supplies the factor \(1/\ell\) and the
contraction produces \eqref{eq:terminal-correction}.  The estimates with
time derivatives used in \cref{sec:regularity} follow in exactly the same
way from \eqref{eq:terminal-profile-bound}; hence the resulting solution is
strong.

The mass and localization proofs of
\cref{prop:localization,prop:mass-concentration} apply after the change of
variables
\[
 z=\frac{x-a}{\mu_\ell(t)}-p,\qquad
 x=a+\mu_\ell(t)(z+p).
\]
They give the asserted mass, blow-up set and edgewise decay.

It remains to identify the parameters.  On any
\(I_\eta=(a-\eta,a+\eta)\) with \(0<\eta<a/2\), differentiation of
\eqref{eq:euclidean-family} and the preceding change of variables give
\[
 \lim_{t\uparrow T}(T-t)
 \norm{\partial_xU_{\ell,p,1}(t)}_{L^2(I_\eta)}
 =\frac1\ell\norm{Q'}_{L^2(\R)},
\]
thus equality up to constant phase forces
\(\ell=\widetilde\ell\).

Assume now that the scales agree.  For every
\(\psi\in C_c^\infty(\R)\), the correction estimate and dominated
convergence yield
\begin{equation}\label{eq:rescaled-density}
 \int_0^\infty
 \psi\!\left(\frac{x-a}{\ell(T-t)}\right)
 \abs{U_{\ell,p,1}(t,x)}^2\dd x
 \longrightarrow
 \int_\R\psi(y)Q(y-p)^2\dd y.
\end{equation}
If two solutions agree up to phase, the left-hand sides agree.  Therefore
\(Q(\,\cdot-p)^2=Q(\,\cdot-\widetilde p)^2\) almost everywhere.  If
\(p\ne\widetilde p\), then the nonzero \(L^1\)-function \(Q^2\) would be
periodic, which is impossible.  Hence \(p=\widetilde p\).
\end{proof}

\subsection*{Remark}

What remains open here is a classification theorem. One would like to determine
whether every minimal-mass solution blowing up on one edge converges, 
up to a constant phase, to one of the profiles in
\cref{prop:terminal-family}, and whether the associated scale and
renormalized center are unique. A possible strategy to address this problem 
is suggested by rigidity arguments for inhomogeneous mass-critical NLS without exact
pseudoconformal symmetry initiated in the seminal paper \cite{RaphaelSzeftel2010}.

\section*{Data availability}
No datasets were generated or analysed
during the current study.

\section*{Statements and Declarations}

\noindent\textbf{Funding.}
M. Buratti acknowledges financial support from
the Istituto Nazionale di Alta Matematica
(INdAM) through a scholarship supporting
his research stay at EPFL.

\end{document}